\documentclass[11pt]{amsart}
\usepackage{amscd}
\usepackage[arrow,matrix]{xy}
\usepackage{graphicx, tikz}
\usepackage{hyperref}
\usepackage{comment}
\usepackage{amsmath}

\usepackage{amsmath, latexsym, amssymb}
\numberwithin{equation}{section}
\theoremstyle{plain}
\newtheorem{lemma}{Lemma}[section]
\newtheorem{proposition}[lemma]{Proposition}
\newtheorem{theorem}[lemma]{Theorem}

\theoremstyle{definition}
\newtheorem{definition}[lemma]{Definition}
\newtheorem*{definition*}{Definition}
\newtheorem{remark}[lemma]{Remark}
\newtheorem{example}[lemma]{Example}
\usepackage{color}
\usepackage{cancel}

\definecolor{brown}{RGB}{150,100,0}

\definecolor{purple}{RGB}{150,0,100}

\definecolor{grey}{RGB}{128,128,128}

\newcommand{\R}{{\mathbb R}}

\newcommand{\E}{{\mathbb E}}

\newcommand{\N}{{\mathbb N}}

\newcommand{\Q}{{\mathbb Q}}

\newcommand{\Dd}{{\mathcal D}}

\newcommand{\Ff}{{\mathcal F}}
\newcommand{\Hh}{{\mathcal H}}

\newcommand{\Jj}{{\mathcal J}}
\newcommand{\Mm}{{\mathcal M}}    

\newcommand{\Pp}{{\mathcal P}}

\newcommand{\Ss}{{\mathcal S}}

\newcommand{\Xx}{{\mathcal X}}
\newcommand{\Yy}{{\mathcal Y}}

\newcommand{\V}{{\mathcal V}}

\newcommand{\eps}{{\varepsilon}}

\newcommand{\g}{{\mathfrak g}}

\newcommand{\diam}{{\rm diam}}
\newcommand{\pb}{{\mathbf p}}
\newcommand{\sign}{{\rm sign}}

\def\NABLA#1{{\mathop{\nabla\kern-.5ex\lower1ex\hbox{$#1$}}}}
\def\Nabla#1{\nabla\kern-.5ex{}_#1}

\newcommand{\p}{{\partial}}

\newcommand{\la}{\langle}
\newcommand{\ra}{\rangle}

\newcommand{\rom}[1]{{\em #1}}

\renewcommand{\)}{\rom)}
\renewcommand{\:}{\colon}

\mathchardef\mhyp="2D

\DeclareMathOperator{\Lin}{Lin}
\DeclareMathOperator{\Hom}{Hom}

\DeclareMathOperator{\MSE}{MSE}

\begin{document}
	\title[Nonparametric  estimations  and the diffeological Fisher metric]{Nonparametric  estimations and the diffeological  Fisher  metric}

	\author[H. V. L\^e]{H\^ong V\^an L\^e}	
	\address{Institute  of Mathematics of the Czech Academy of Sciences,
	Zitna 25, 11567  Praha 1, Czech Republic}
\email{hvle@math.cas.cz}

\author[A. A. Tuzhilin]{Alexey A. Tuzhilin}
\address{Moscow  State  University  Lomonosov,   Faculty  of Mechanics and  Mathematics, Moscow, Russia, 119991}
\email{tuz@mech.math.msu.su}


\thanks{Research  of HVL was supported  by  GA\v CR-project 18-01953J and	 RVO: 67985840}

\begin{abstract}
	In  this     paper,  first, we      survey  the concept   of   diffeological Fisher  metric  and its naturality,  using    functorial language   of    probabilistic morphisms, and slightly  extending L\^e's theory in \cite{Le2020} to include  weakly $C^k$-diffeological  statistical  models.   Then we introduce    the resulting   notions of  the diffeological    Fisher  distance,  the  diffeological Hausdorff--Jeffrey measure  and    explain their   role   in   classical   and Bayesian nonparametric  estimation problems  in statistics.
	
\end{abstract}

\maketitle

\section{Introduction}\label{sec:intr}

  In  the present paper    we  survey    the  concept  of the diffeological   Fisher  metric, introduced  in  \cite{Le2020},  and explain its role in   frequentist and Bayesian  {\it  nonparametric}  density estimations.   Diffeological Fisher  metric  is   a natural extension  of the   Fisher  metric  to    singular  statistical models,  which are ubiquitous  in  machine  learning \cite{Watanabe2009}, \cite{Amari2016}.   Among   different  approaches  to  singular spaces,   we find the Souriau  theory  of diffeological spaces \cite{Souriau1980}   best suitable  for   our study  of statistical models,   parameterized statistical models  and  dynamics  on them.     The role  of   the  diffeological  Fisher  metric  in   frequentist nonparametric   estimation is expressed  via the  Cram\'er--Rao inequality (Theorem  \ref{thm:cr}, Remark \ref{rem:cr}).  The   role  of the  diffeological Fisher  metric  in Bayesian   estimations  is expressed   via  the          choice  of the objective  a prior     Hausdorff--Jeffrey   measure  on  2-integrable diffeological statistical models (Definition \ref{def:hausdorfj}, Theorem \ref{thm:jeffrey}).    The Hausdorff-Jeffrey    measure   is a natural  generalization of the    Jeffrey measure,  using  the     concept  of the  diffeological Fisher  distance  that is introduced  in the  present  paper,  and combining  with  the  concept  of the Hausdorff measure  in geometric measure  theory.   Geometric measure   theory could be described  as differential geometry, generalized through measure  theory to deal   with   singular mappings and  singular spaces  and applied to the  calculus of variations \cite{Federer1969}, \cite{AT2004}, \cite{Morgan2009}. Hausdorff measures play an important role in several areas of mathematics, e.g., in the theory of fractals, in the theory of stochastic processes. We  also  refer  the reader to \cite{JLT2021}   for a categorical   treatment  of the   Dirichlet (a prior) measure on  the set $\Pp (\Xx)$ of all probability  measures  on a  measurable space $\Xx$ whose  $\sigma$-algebra  will be denoted by $\Sigma_\Xx$.

The remaining part  of  our  paper is organized as follows. In Section \ref{sec:2} we  recall the concept of a $C^k$-diffeological space, adapted  from \cite{IZ2013},  and the resulting concepts  of   a  $C^k$-diffeological  statistical model and a weakly $C^k$-diffeological statistical model (Definitions \ref{def:diff}, \ref{def:wtangent}, Example \ref{ex:weakc1},  Lemma \ref{lem:dominated}, Remarks \ref{rem:weak}, \ref{rem:wfisher}), the  notion of the diffeological  Fisher  metric,  slightly extending   the   concepts  introduced by L\^e in \cite{Le2020}.  Then  we introduce  the  notion of the   diffeological Fisher  distance (Definition \ref{def:length}, Theorem \ref{thm:distance}, Remark \ref{rem:wdist}). In the last part  of Section \ref{sec:2} we recall the concept  of   probabilistic  morphisms  and     the monotonicity (resp.  the invariance) of the  diffeological  Fisher metric    under   probability  morphisms     (resp. sufficient   probabilistic  morphisms). Then  we  deduce   similar   functorial   properties  for  the diffeological  Fisher distance.  In  Section \ref{sec:3}  we   recall  the concept  of  a nonparametric
$\varphi$-estimator introduced  in \cite{Le2020}  and   the  related  diffeological  Cram\'er--Rao inequality, proved in \cite{Le2020}, see  also Remark \ref{rem:cr}, where   we discuss  the validity  of the   diffeological Cr\'amer--Rao for  weakly  $C^k$-diffeological statistical models.  In   Section \ref{sec:hausdorfj}  we  introduce  the  resulting notion  of the Hausdorff--Jeffrey measure (Definition \ref{def:hausdorfj}).  Then   we  derive  their    monotonicity and invariance property    from the   corresponding properties    of the Fisher  distance (Theorem \ref{thm:jeffrey}). In the last section  we  discuss  some   open questions and future directions.

\section{Diffeological  Fisher metric, diffeological  Fisher  distance  and probabilistic morphisms}\label{sec:2}

First  let us recall the notion of a   $C^k$-diffeological space.

\begin{definition}\label{def:diff} {\cite[Definition 3]{Le2020}, cf. \cite[\S 1.5]{IZ2013}}
For $k \in\N \cup \infty$  and   a nonempty set $X$,   a {\it $C^k$-diffeology} of $X$ is a set $\Dd$ of   mappings  $\pb: U \to  X$,  where $U$ is an open domain  in $\R^n$, and $n$   runs over  nonnegative integers,  such that the three following
	axioms are satisfied.
	
	D1. {\it Covering}. The set $\Dd$ contains the constant mappings ${\bf x}:   r\mapsto x$, defined on $\R^n$, for all $x \in X$  and  for all $n \in \N$.
	
	D2. {\it  Locality}. Let ${\pb}: U \to X$ be a mapping. If for every
	point $r \in U$ there exists  an open neighborhood  $V $ of $r$ such that  ${\pb}_{| V}$ belongs to  $\Dd$ then  the map $\pb$ belongs to $\Dd$.
	
	D3. {\it Smooth compatibility}. For every element $\pb: U \to  X$ of $\Dd$, for every real
	domain $V$, for every  $\psi \in C^k(V, U)$, $\pb \circ \psi$ belongs to $\Dd$.
	
	A {\it $C^k$-diffeological space}  $X$ is a nonempty set $X$ equipped with a $C^k$-diffeology  $\Dd$. Elements  $\pb: U \to X$  of $\Dd$   will be called {\it  $C^k$-maps  from $U$ to $X$}.
	
	 A  map  $f:  (X, \Dd) \to ( X', \Dd')$ between    two  $C^k$-diffeological  spaces    is  called {\it a $C^k$-map}, if  for  any  $\pb \in \Dd$   we have $ f  \circ \pb \in \Dd'$.
\end{definition}

\underline{Digression.} Recall  that    a map
$\varphi: U \to  V$  is  called {\it weakly    \(Fr\'echet\/\)\footnote{in this paper we shall consider  only (possibly weakly) Fr\'echet differentiable mappings and we shall  omit  ``Fr\'echet" in the  remaining part of this paper}  differentiable   in $u_0 \in U$}  if  there exists  a bounded  linear  operator  $ d\varphi_{u_0}: \R^n \to V$  such that \cite[p. 384]{AJLS2017}
$$w\mhyp\lim_{ v \to 0}  \frac{\varphi (u_0  + v) - \varphi(u_0) -  d\varphi_{u_0}  (v)}{|| v||} = 0 ,$$
where  $w\mhyp\lim$ denotes  the weak  limit. In this case  $d\varphi_{u_0}$  is called  {\it the weak differential  of $\varphi$  at $u_0$}. Denote by $\Lin (E, V)$ the  Banach  space  of bounded  linear maps from  a Banach  space  $E$ to  a Banach  space $V$  with the induced  norm.  A  map  $\varphi: U \to V$  is called {\it a  weak $C^k$-map},  if it is  weakly differentiable,  and if the inductively defined maps  $d^1: = d\varphi:  U \to Lin (\R^n, V)$,  and
$$d^{r+1} \varphi: \R^n \to \Lin  ((\R^n)^r,  V),  u \mapsto d (d^r\varphi)_u  \in \Lin ((\R^n)^{r-1}, V) $$
are  weakly differentiable  for  $r = 1, \cdots, k -1$ and weakly continuous  for $r =k$  \cite[p. 384]{AJLS2017}.
Clearly  the composition  of   weak $C^k$-maps  is a    weak $C^k$-map  and   a   weak $C^k$-map  between
finite dimensional smooth manifolds   is a  $C^k$-map.  We    also   write  shorthand  ``$w\mhyp C^k$-map"  for ``weak $C^k$-map".

 The concept  of   a  weak $C^k$-map  is  a   natural   extension of the  concept of  weak convergence.  The  weak convergence  of measures  is  one of  most important  tools in  applied and theoretical statistics  \cite{Bogachev2018}.  It is known that  the class of  weakly  differentiable maps  is   strictly larger  than  the class  of  differentiable  maps \cite{Kaliaj2016}.

\begin{example}\label{ex:wcan}  (1) Let $ V$  be a Banach  space.  Then  $V$  has the   canonical   $C^k$-diffeology $\Dd_{can}^k$  that consists of all   $C^k$-mappings   $\pb: U \to V$, where  $U$  is an open domain in $\R^n$.     The  space $V$   has  also  another  $C^k$-diffeology  $\Dd_{w}^k$  that consists  of  all  weak $C^k$-mappings  $U \to V$, where  $U$ is an open  domain  in $\R^n$.

(2)  Assume that   $(X', \Dd')$  is a   $C^k$-diffeological   space  and  $ f: X \to X'$ is  a   map.  Then
 the {\it pullback  diffeology} $f^* (\Dd')$  is    the $C^k$-diffeology on $X$  defined as follows \cite[p. 14]{IZ2013},
 $$ f^* (\Dd') : = \{  \pb: U \to X|\,  f \circ  \pb \in \Dd'\},$$
 where  $U$ is an open subset of $\R^n$.

 (3) Let  $(X, \Dd)$  be a  $C^k$-diffeological  space  and  $f: X \to X'$   a map. Then   the  {\it pushforward diffeology}
 $f_*(\Dd)$  is  the  diffeology  on  $X'$ that consists  of all   mappings $\pb: U \to X'$ where  $U \subset \R^n$ is  an open  subset   and $\pb$ satisfies the following property  \cite[p. 24]{IZ2013}.  For every  $ u \in U$,   there exists  an open neighborhood  $O (u)\subset U$  of $u$ such that, either  $\pb|_{O(u)}$  is a constant  map, or  there exists  a map ${\bf q}: O(u) \to   X$  such that $\pb |_{O(u)} =   f\circ {\bf q}$.
\end{example}

Let us   recall  that  the space $\Ss(\Xx)$ of  all  finite signed  measures on a measurable  space $\Xx$ is   a Banach space, denoted by $\Ss(\Xx)_{TV}$,
with the  total  variation norm $\| \cdot \| _{TV}$. For any statistical model $\Pp_\Xx$, which is, by definition,  any subset in  $\Pp(\Xx)\subset\Ss(\Xx)$ \cite{McCullagh2002}, \cite{Le2020},   we denote by $i: \Pp_\Xx \to \Ss(\Xx)$ the natural  inclusion.

\begin{definition}\label{def:wstat} {cf. \cite[Definition 3]{Le2020}}
(1)	A   statistical model  $\Pp_\Xx$  endowed
	with a $C^k$-diffeology $\Dd_\Xx$  is called  a {\it $C^k$-diffeological statistical  model} or a  {\it weakly  $C^k$-diffeological  statistical model\/}, respectively,  if    $i_*(\Dd_\Xx)  \subset \Dd_{can} ^k$  or
	  $i_*(\Dd_\Xx)  \subset \Dd_{w} ^k$, respectively.  A    $C^k$-diffeology on  a   weakly $C^k$-diffeological statistical   model will be called {\it a weak  $C^k$-diffeology}.
	
(2)	 Let $\Dd_\Xx$ be a   $C^k$-diffeology  on  a statistical model $\Pp_\Xx$. For $l \in \N \cup \infty$  we shall  call an element  $\pb: U \to \Pp_\Xx$ in $\Dd_\Xx$   of  {\it class  $C^{k+l}$}  or  of   {\it  class $w\mhyp C^{k+1}$}, respectively, if $i \circ  \pb \in \Dd_{can} ^{k+l}$  or $i \circ  \pb \in \Dd_{w} ^{k+l}$, respectively.  In other words,  there is a filtration  of  diffeologies  $(i ^* (\Dd^\infty_{can}) \cap \Dd_\Xx) \subset \cdots   \subset (i^*(\Dd^k_{can})\cap \Dd_\Xx )= \Dd_\Xx$  or $(i ^* (\Dd^\infty_{w}) \cap \Dd_\Xx ) \subset \cdots   \subset (i^*(\Dd^k_{w}) \cap \Dd_\Xx) = \Dd_\Xx$), respectively.
\end{definition}

 Examples  of
$C^k$-diffeological  statistical models are      the image  $(\pb(M), \pb _* (\Dd_{can}^k))$ of     parameterized   statistical models  $(M, \Xx, \pb)$,   where   $M$ is a   smooth    Banach manifold  and  $i \circ \pb: M \to \Ss(\Xx)$  is a  $C^k$-map,  see \cite[Example 8.2]{Le2020}. There  are  many  parameterized   statistical models $(M, \Xx, \pb)$  whose  image  $\pb (M)$ are  singular  statistical models \cite{Amari2016}, \cite{Watanabe2009}, see  also  Example \ref{ex:normalm} below.  We shall    provide  an  example  of an weakly  $C^1$-diffeological  statistical model, which  is not   a $C^1$-diffeological  statistical model.
 \begin{example}\label{ex:weakc1}  Let  $\Xx = [-\pi, \pi]$   with  the Lebesgue measure  $dx$.
 For  $ t \in [-1, 1]\setminus \{0\}$  we   set
 $$ f_t (x) = \sin (\frac{x}{t})$$
 and we   let   $f_0 (x) =0$. Then  for  all $t$  we have  $f_t  \in  L^1 (\Xx, dx)$  and
 $f_t$   is weakly  continuous  in $L^1 (\Xx, dx)$  but not  strongly continuous.
 Next we  define   a function $F_t (x)$  for $ t \in (-1, 1)$  and $ x\in [-\pi, \pi]$  as follows.
 $$F_t (x) : = \int_0 ^t f_s (x)ds.$$
 Since   $f_s(x) = - f_s(-x)$, for  all $t \in [-1, 1]$
 \begin{equation}\label{eq:p}
 \int _{-\pi}^\pi  F_t  (x) \,dx  = 0.
 \end{equation}
 Since  $F_t (x)$ is continuous  in  $t$ and  in $x$  there  exists  a  number
 $A>0$  such that
 $$2 \pi |F_ t (x) | \le  A \text{ for all } (t, x) \in [-1, 1]\times [-\pi, \pi].$$   Finally we define a  map $c : (-1, 1) \to  \Pp (\Xx) \subset \Ss (\Xx)$
 $$  c(t) : =  \big(\frac{1}{2\pi} +  \frac{F_t (x)}{2A}\big) d x.$$
 Clearly $c (t)$ is   differentiable, but its derivative  $c' (t) = {1\over  2A}f_t(x)  dx$   is  only  weakly continuous,  therefore  the   map $c$ is  a weak $C^1$-map  but not  a $C^1$-map.  Hence
 the image   of $c$  is a weakly  $C^1$-diffeological  statistical model, which  is not   a $C^1$-diffeological  statistical model.
 \end{example}
  Concerning  weakly  $C^k$-diffeological statistical models we have the following local structure result.

Let us recall  that   a finite  signed  measure  $\nu\in \Ss(\Xx)$  is said  to be  dominated by   a non-negative measure   $\mu$ on $\Xx$, if     $\mu (A) = 0$  implies $\nu (A) = 0$ for any $A \in \Sigma_\Xx$. Alternatively, $\nu$ is  called absolutely continuous w.r.t. $\mu$, see e.g. \cite[Chapter IV]{Neveu1970}.

\begin{lemma}\label{lem:cdominated} {cf. \cite[Proposition 3.3, p. 150]{AJLS2017}} Assume that $U \subset \R^n$ is an open connected domain and $\varphi: U \to \Pp (\Xx)$ is a   map such   that $i \circ \varphi: U \to \Ss(\Xx)$  is a weak $C^1$-map. Then  there exists  $\mu_0 \in \Pp (\Xx)$  that dominates $\varphi (u)$  for all $u \in U$.
\end{lemma}
\begin{proof}    Let  $U_\Q \subset U$  be  the subset  of all  points in $U$ with rational coordinates in $\R^n$. Then $U_\Q$  is a  countable  set. By \cite[Lemma 3.1, p. 146]{AJLS2017}, cf. \cite[Ex.IV.1.3]{Neveu1970} there  is a measure $\mu_0 \in \Pp (\Xx)$  that  dominates $\varphi (u)$  for all $u \in U_\Q$. Now let  $u \in U$.  We  shall prove  that $\varphi (u)\ll  \mu_0$. Assume  that
$A \in \Sigma_\Xx$  is a null-set  of  $\mu_0$.
Then   for all  $k$ we  have  $\varphi(u_k)  (A) =0$.  Since  $\varphi: U \to \Ss(\Xx)$ is weakly continuous, and
$u_\Q$ is dense  in $U$, it follows  that  $\varphi (u) (A) =0$. Hence  $\varphi(u) \ll \mu_0$. This  completes
the  proof  of Lemma \ref{lem:dominated}.
\end{proof}

The concept of   the tangent  space of a  $C^k$-diffeological  statistical  model $(\Pp_\Xx, \Dd_\Xx)$  at a point $\xi \in \Pp_\Xx$ \cite[Remark 2]{Le2020}) extends naturally to  the case  of  weakly $C^k$-diffeological  statistical models, see  Definition \ref{def:wtangent} below. We also   refer the reader  to  \cite[(5.1)]{Souriau1980},  \cite[p. 166]{IZ2013} for  a bit more abstract  approach.
Note   that any  $C^k$-diffeological statistical model  is a  weakly $C^k$-diffeological  statistical  model.

\begin{definition}\label{def:wtangent}{cf.  \cite[Remark 2]{Le2020}}  Let  $(\Pp_\Xx, \Dd_\Xx)$  be a  
weakly  $C^k$-diffeological statistical model.  Let  $ c:  (-\eps, \eps) \to  (\Pp_\Xx, \Dd_\Xx)$  be a   $C^k$-map.  The  
{\it tangent  vector}  $\p_t c(0)$  at  $c(0)$ is the image  of the  map $dc_{0}   (\p t)\in \Ss(\Xx)$, where $dc_0$ is 
the weak  differential of  $c$ at $0$. For  $\xi\in  \Pp_\Xx$, the
tangent  cone $C_\xi (\Pp_\Xx, \Dd_\Xx)$  consists  of  all
 tangent   vectors  $\p_t c(0)$ at $c(0) = \xi$, where $c: (0,1) \to (\Pp_\Xx, \Dd_\Xx)$ be a $C^k$-map, and the  tangent  space $T_\xi  (\Pp_\Xx, \Dd_\Xx)$  is the linear  hull of
$C_\xi (\Pp_\Xx, \Dd_\Xx)$.
\end{definition}

\begin{lemma}\label{lem:dominated}  Let $v$  be a  tangent  vector   at $\xi$ in  a weakly  $C^k$-diffeological  statistical model  $(\Pp_\Xx, \Dd_\Xx)$. Then $v$ is dominated  by  $\xi$.
\end{lemma}
\begin{proof} The  proof  of  Lemma  \ref{lem:dominated} uses  the same  argument in the  proof    for  the case of tangent vectors  of   $C^k$-diffeological  statistical models  \cite[Remark 2]{Le2020}, \cite[Corollary 3.3.2, p.77]{Bogachev2010},  \cite[Theorem 3.1, p. 142]{AJLS2017}.   Let $v = \p_t c(0)$, where $c: (-\eps, \eps) \to (\Pp_\Xx, \Dd_\Xx)$ is a  weak  $C^k$-map.  Let $A \in \Sigma_\Xx$  such that $c(0)  (A) = 0$.   Since the map  $ I_A: \Ss(\Xx) \to \R, \:  \mu \mapsto  \mu (A),$  is a linear  bounded  map, the map   $ I_A \circ c:   \to \R$ is a  $C^1$-map, see e.g.  \cite[Proposition C.2, p. 385]{AJLS2017}. It follows that
$$ \frac{d}{dt}_{| t =0}  I_A\circ  c(t) = I_A (v) =0$$
since  $I_A \circ  c(t) \ge  0$. Hence $ v\ll  c(0)= \xi$. This completes  the proof  of  Lemma \ref{lem:dominated}.
\end{proof}
  Lemma  \ref{lem:dominated}  implies  that  for  any tangent vector $v$   at   a point $\xi$  of a   weakly $C^k$-diffeological statistical model $(\Pp_\Xx, \Dd_\Xx)$,
{\it  the logarithmic representation of $v$}
\begin{equation}
\label{eq:logrep}
\log v: =dv/d\xi
\end{equation}
is an element  of $L^1(\Xx,\xi)$.
The  set   $\{ \log  v|\:  v\in C_\xi (\Pp_\Xx, \Dd_\Xx)\}$  is a  subset  in
$L^1(\Xx, \xi)$. We denote      it   by    $\log (C_\xi (\Pp_\Xx, \Dd_\Xx))$ and  will  call  it {\it the logarithmic representation   of $C_\xi (\Pp_\Xx, \Dd_\Xx)$}. In  \cite[Definition 3.6, p. 152]{AJLS2017}, for a   $C^1$-map $c : (0,1)\to \Pp_\Xx \subset \Ss(\Xx)$  we call  $dc (\p t)/dc(t)$ the logarithmic derivative  of $c$ in the direction  $\p t\in T_t (0,1)$, since  in the classical  case where  $c (t)  =   f(t) \cdot \mu_0$  is a  dominated measure family with differentiable  density  function $f(t)$, then
$ dc (\p t ) / d c(t)  =   (d/dt)\log   f(t)$.

 \begin{remark}\label{rem:weak}
 Any bounded  function  $H$ on  $\Xx$  defines  a  continuous   linear   function  $I_H$  on  the Banach  space $\Ss(\Xx)_{TV}$ as follows
 $$I_H : \Ss(\Xx)_{TV} \to \R, \: \mu \mapsto \int_\Xx H d\mu.$$
 Assume  that     a map  $\varphi: (0,1)  \to  \Pp (\Xx), \:  t \mapsto \mu_t,$   is weakly differentiable.  Let  $\mu'_t : = \p_t (\varphi (t)) \in \Ss (\Xx)$. Then  we have
 \begin{equation}\label{eq:exchange}
 {d\over dt|}_{t =0}\int_\Xx H d\mu_t  = \int _\Xx  H d (\mu'_0).
 \end{equation}
The  identity  (\ref{eq:exchange})  is  central  for many applications, see e.g.  \cite{Pflug1996} and Remark \ref{rem:cr},  and  therefore  the concept of  weakly $C^k$-diffeological  statistical  models is useful.
Note that measure  valued  weak differentiable  maps  from an open  subset of $\R^n$ have been  first introduced  by  Pflug  \cite{Pflug1988}, see  also \cite[Definition 3.25, p. 158]{Pflug1996}   in the case  $\Xx$ is   a metric  space with Borel $\sigma$-algebra, using (\ref{eq:exchange})  as the definition (with  $H$  bounded and continuous).
 \end{remark}

\begin{definition}\label{def:2-integrable}{\cite[Definition 4]{Le2020}}
A $C^k$-diffeological statistical model $(\Pp_\Xx,\Dd_\Xx)$ will be called {\it almost $2$-integrable}, if $\log\bigl(C_\xi(\Pp_\Xx,\Dd_\Xx)\bigr)\subset L^2(\Xx,\xi)$ for all $\xi\in\Pp_\Xx$. In this case we define {\it the diffeological Fisher metric $\g$ on $\Pp_\Xx$} as follows. For each $v,w\in C_\xi(\Pp_\Xx,\Dd_\Xx)$  we set
	
\begin{equation}\label{eq:rFisher}
\g_\xi(v,w):=\la\log v,\log w\ra_{L^2(\Xx, \xi)}=\int_\Xx\log v\cdot\log w\,d\xi.
\end{equation}

The  Fisher  metric  on  $C_\xi (\Pp_\Xx,\Dd_\Xx)$  extends naturally to  a  positive quadratic form
on $T_\xi (\Pp_\Xx, \Dd_\Xx)$, which  is also  called   the Fisher metric.
	
An almost 2-integrable  $C^k$-diffeological  statistical model $(\Pp_\Xx, \Dd_\Xx)$ will be  called {\it 2-integrable},  if   for  any $C^k$-map $\pb: U \to P_\Xx$ in $\Dd_\Xx$,
the function  $v\mapsto\bigl|d\pb(v)\bigr|_\g$ is continuous  on $TU$. 
\end{definition}

\begin{remark}\label{rem:wfisher}  (1) As in Definition \ref{def:2-integrable}, we  shall   say  that  a   weakly $C^k$-diffeological statistical model  is  almost 2-integrable,   if   we can define  the Fisher metric  on its  tangent     cone as in (\ref{eq:rFisher}). We shall say  that   an almost  2-integrable  weakly $C^k$-diffeological  statistical  model   $(\Pp_\Xx, \Dd_\Xx)$ is {\it 2-integrable},  if   for  any weak $C^k$-map $\pb: U \to \Pp_\Xx$ in $\Dd_\Xx$, the function $v \mapsto |d\pb(v)|_\g$ is continuous  on $TU$.
	
(2) On  $C^k$-diffeological spaces, in particular on (weakly) $C^k$-diffeological   statistical  models  $(\Pp_\Xx, \Dd_\Xx)$, we can define  the notion of  $C^k$-functions.  If the  dimension of its  tangent  spaces $T_\xi (\Pp_\Xx, \Dd_\Xx)$ is finite  for all $\xi \in \Pp_\Xx$, then we can define  the  notion of   a gradient  of a $C^k$-differentiable  function on $(\Pp_\Xx, \Dd_\Xx)$.
\end{remark}

\begin{example}\label{ex:normalm}  Let us consider    an example  of  a   2-integrable  $C^\infty$-diffeological    statistical model   which is  the image of    a   parameterized  statistical model   $(W, \R, \pb = p \cdot \mu_0)$  where
	$$W =\{ ( a, b) \in \R^2 |\,  a \in [0,1],  b \in \R \}, $$
	 $\mu_0$ is the Lebesgue measure  on $\R$,  and
	$$ p ( x|  a, b ) : = \frac{(1 -a)  e^{ - x^2 /2}  + a   e ^{ - (x -b) ^2 /2}}{\sqrt {2\pi}} .$$
	This  family is a typical  example   of  Gaussian mixture models  \cite[Example 1.2, p. 14]{Watanabe2009}, which comprise  also  the changing time  model  (the Nile River model) and the ARMA model  in time series \cite[\S 12.2.6, p. 311]{Amari2016}.  We   decompose  $W$ as    a disjoint union   of its  subsets
	as follows
	$$  W = W_{-} \cup W_0 \cup W_+$$
where
\begin{gather*}
W_-=\bigl\{(a,b)\in W\mid a\in(0,1),\,b<0\bigr\},\\
W_0=\bigl\{(a,b)\in W\mid a\in(0,1)\ \&\ b=0\text{ or }a=0\ \&\ b\in\R\bigr\},\\
W_+=\bigl\{(a,b)\in W\mid a\in(0,1),\,b>0\bigr\}.
\end{gather*}
The restriction of $\pb$ to $W_-\cup W_+$ is injective, and $\pb(W_0)=\pb(0,0)$. We compute
\begin{gather*}
\p_ap(x|a,b)=\frac{-e^{-x^2/2}+e^{-(x-b)^2/2}}{\sqrt{2\pi}},\\
\p_bp(x|a,b)=\frac{a(x-b)e^{-(x-b)^2/2}}{\sqrt{2\pi}}.
\end{gather*}

In \cite[p. 14]{Watanabe2009},  using the  expression of the Fisher metric    via the  Kullback-Leibler divergence,  Watanabe  showed  that  the Fisher  metric  on   the  parameterized  statistical model $(W, \R, \pb = p \cdot \mu_0)$    exists and  is continuous. It  follows that  $(\pb_*(W), \pb_*(\Dd^k_{can}))$  is a 2-integrable     $C^k$-diffeological statistical model  for any $k \in \N$.
\end{example}

\begin{lemma}\label{lem:incl}
The     statistical    model $\pb_*(W)$ has  two different    $C^1$-diffeologies   $\pb_* (\Dd_{can}^1)$  and  $i^*(\Dd^1_{can})$.
\end{lemma}
\begin{proof} Since  $\p _a p (x | 0, 0) = \p_b p (x|0,0)$  we have  $ T_{\pb (0,0)} (\pb (W), \pb_* (\Dd_{can}^1))= \{0\}$.
Now   we shall  show  that  $T_{\pb (0,0)} (\pb (W), i^*(\Dd_{can} ^1))$  contains   a nonzero  vector.
Let  us  consider  a $C^1$-curve $c: (-1, 1) \to \pb (W) \stackrel{i}{\to}  \Ss (\R)$ defined  as follows
\begin{equation*}
c(t) : = \frac{ (1 - \alpha(t)) e^{- x^2 /2} +  \alpha (t) e ^{- ( x - \beta  (t)) ^2 /2}
}{\sqrt{2\pi}},
\end{equation*}
where  $\alpha (t), \beta(t)$  are   the following   functions  on $(-1, 1)$:
\begin{eqnarray*}
\alpha  (t) = \int_0 ^ t\frac{ d\tau}{\log (\tau ^2)} \text{  for }  t \not= 0  \text{ and }  \alpha (0) = 0,\\ 
\beta(t)=  t \log (t^2) \text{ for }   t \not = 0 \text{ and  } \beta (0) = 0.
\end{eqnarray*}
Clearly  $\alpha, \beta$  are   continuous functions. Moreover
$\alpha$  is a $C^1$-function  and $\beta$  is  a $C^1$-function  outside the point  $0 \in (-1, 1)$  and
\begin{equation*}
\dot c(t) = \frac{-\dot \alpha(t) e ^{ - x^2 /2} + \dot \alpha (t) e^{- (x - \beta(t))^2 /2} + \alpha (t) (x- \beta(t))\dot \beta(t) e^{- (x-\beta(t))^2/2}}{\sqrt{2\pi}}
\end{equation*}
Since
$$
\dot\alpha(t)=\frac{1}{\log(t^2)},\ \ \dot\beta(t)=\log(t^2)+\frac{2t^2}{\log(t^2)},
$$
 we have
 $$\lim_{ t \to 0}  \dot  c (t) =  \frac{ x e ^{-x^2 /2}}{\sqrt{2\pi}}.$$
 This  implies that  $ c(t)$ is a   $ C^1$-curve in   $\bigl(\pb(W),i^*(\Dd^1_{can})\bigr)$   and   $ c(0) = 0,  \dot   c(0)  \not = 0.$  This completes  the proof  of Lemma \ref{lem:incl}.
\end{proof}

\begin{example}\label{ex:friedrich}   Let $\Xx$ be a measurable space   and   $\lambda$  be    a $\sigma$-finite measure. In \cite[p. 274]{Friedrich1991} Friedrich considered
a  family  $P(\lambda):=\{  \mu \in \Pp(\Xx)|\, \mu \ll   \lambda\}$  that is endowed  with   the following  diffeology  $\Dd (\lambda)$.   A  curve $c: \R \to  P (\lambda)$ is a   $C^1$-curve, iff
$$\log\dot  c (t) \in  L^2 (\Xx, c(t)).$$
Hence $(P(\lambda),\Dd(\lambda))$ is     an  almost 2-integrable
$C^1$-diffeological statistical model, see  \cite[Example 10]{Le2020}.    Next we shall prove  that
$P(\lambda)$ is not a 2-integrable   $C^1$-diffeological   statistical model  for   $\Xx =  (-1, 1)$  and  $\lambda$  being the  Lebesgue  measure  $dx$.  It suffices  to show
a   $C^1$-curve $c:  (-1, 1) \to P (dx)$  such  that  $\dot  c (t) \in  L_2 (\Xx, c(t))$  for all $t \in (-1, 1)$  but  $|\dot c (t)| _\g$   is not continuous    at $ t =0$.    We shall  construct  such a curve  using \cite[Example 3.4, p.  155]{AJLS2017}. 
First  we consider  a  smooth function
$f: [0,\infty) \to \R$     such that  
$$ f(u) >0, f' (u) < 0 \text{ for } u \in [0,1), \text{ and }   f (u)  = 0 \text{  for } u \ge 1.   $$
For  $t\in  (-1, 1)$ we define  $\pb :  (-1, 1) \to \Ss (\Xx), \: \pb  (t)  = p (t, x) dx, $  where
$$p (t, x) : = \begin{cases}
1  &  \text{ if }   x \le 0 \text{ and } t \in \R\\
|t|  f (\frac{x}{t})^2  \, dx   &  \text{ if }   x >0 \text{ and } t \not = 0   \\
0 &  \text{ otherwise }
\end{cases}
$$
Then  for all $ t \in (-1, 1)$  we have $\pb (t) (\Xx)= \| \pb (t)\|_{TV} \ge 1 $. By  op. cit.,
\begin{equation}\label{eq:est1}
| \pb (t) - \pb (0)|| _{TV} =  t^2 \int _0 ^1 f(u)  ^2  \, du \le    A
\end{equation}
for some   finite  constant  $A$, hence   $\|\pb (t)\|_{TV} \le 2A$  for all $t \in (-1, 1)$.  It has been   shown  ibid. that   $\pb: (-1, 1) \to \Ss(\Xx)$ is a    $C^1$-map.
Now  we set
$$  c(t) : = \frac{\pb (t)}{\| \pb (t)\|_{TV}}.$$
Then  $c: (-1, 1)  \to \Ss (\Xx)$   is  a $C^1$-curve lying on $\Pp (\Xx)$. By local. cit. we have
for  $t \not = 0$
$$\dot \pb (t) =  \chi_{(0,1)}(x) \sign  (t)   \Bigl (  f(u) ^2  -   2 u f f ' (u)\Bigl)|_{ u = x | t|  ^{-1} }  dt $$
and $\dot \pb (0) =0$.
It follows  that $\dot \pb (t) \in L^2 (\Xx, \pb (t))$. Furthermore we have
\begin{gather*}
\dot c(t)=\frac{\dot\pb(t)}{\|\pb(t)\|_{TV}}+\frac{\pb(t)(d/dt)\|\pb(t)\|_{TV}}{\|\pb (t)\|_{TV}^2},\\[5pt]
\|\dot c_t\|_\g^2=\int_\Xx\,\biggl|\frac{\dot c(t)}{c(t)}\biggr|^2c(t)\,dt=
\frac{1}{\|\pb(t)\|_{TV}}\int_\Xx\,\biggl|\frac{\dot\pb(t)}{\pb(t)}+
\frac{(d/dt)\|\pb(t)\|_{TV}}{\|\pb(t)\|_{TV}}\biggl|^2\pb(t)\,dt<\infty.
\end{gather*}
Thus   $\dot  c(t) \in L^2 (\Xx,   c(t))$  for all $ t \in (-1, 1)$  and  $\dot c (0) = 1$.
Since  $\lim _{t \to 0}\|\dot \pb (t)\|_{TV} = 0$, it follows
$$\lim_{t\to 0} \frac{ d} {dt}\| \pb (t)\|_{ TV } = 0.$$
Since $\|\pb (0)\|_{TV} = 1$, it  follows  that
$$\lim_{t \to 0}   \| \dot c_t\| _g  = \lim _{t \to 0} \| \dot \pb (t)\|_\g $$	which is  positive  by Ay-Jost-L\^e-Schwachh\"ofer's result loc. cit. This    proves  our claim  that   $P (dx)$  is  not a 2-integrable $C^1$-diffeological statistical model.
\end{example}	
	
We shall    use the  diffeological  Fisher metric  to define   the  Fisher   distance  on
 2-integrable      $C^k$-diffeological statistical models  $(\Pp_\Xx, \Dd_\Xx)$. Recall that
 $\Pp_\Xx$ is a topological  space  with   the strong topology  induced from the strong  topology  on the Banach space $\Ss(\Xx)$.

\begin{definition}\label{def:length}   Let $ (\Pp_\Xx, \Dd_\Xx)$ be  a  2-integrable   $C^k$-diffeological statistical  model.

(1) A   map $c: [a, b] \to   (\Pp_\Xx, \Dd_\Xx) $   will be  called  {\it   a   $C^k$-curve},  if   there  exists  $\eps >0$   and   a   $C^k$-map: $c_\eps: (a-\eps, b + \eps) \to  (\Pp_\Xx, \Dd_\Xx)$ such that  the restriction of $c_\eps$ to $[a,b]$ is   $c$.

(2)  A continuous map $ c: [0,1] \to (\Pp_\Xx, \Dd_\Xx)$   will be   called {\it  a piece-wise   $C^k$-curve}, if there  exists  a finite   number  of points   $0=  a_0< a_1 < a_2 \cdots  < a_m = 1$  such  that
 the  restriction  of $c$  to $[a_{i-1}, a_i]$ is   a     $C^k$-curve  for $i \in [1,m]$.

(3)  Let  $c:  [0,1]\to  (\Pp_\Xx,\Dd_\Xx)$   be  a  $C^k$-curve  connecting    $q_1, q_2 \in \Pp_\Xx$ such that $c(0) = q_1 $ and $c(1)= q_2$.   We define {\it the  length}  of   $c$  by
 $$L (c) =\int_0 ^1 |\dot  c (t)|_\g\, dt $$
where $|\cdot |_\g$  denotes  the length  defined  by the  diffeological Fisher metric $\g$.
 {\it The  length} of a piece-wise   $C^k$-curve   will be defined  as the  sum of the  lengths  of  its
 $C^k$-smooth  sub-intervals.

(4)  The     diffeological  Fisher  distance   $\rho_\g (x, y)$ between   two points  $x, y \in \Pp_\Xx$  will be defined as    the  infimum   of the   length  over   the      space of piece-wise $C^k$-curves  connecting    $x, y$. In particular,  if there is  no $C^k$-path  connecting $x, y$ then  $\rho_\g (x, y) = \infty$.
\end{definition}

\begin{theorem}\label{thm:distance}
The distance function $\rho_\g(x,y)$ is an extended metric, i.e., it can be infinite somewhere.
\end{theorem}

\begin{proof}      Clearly  $\rho_\g (x, y)$ is a   symmetric  nonnegative function and   $\rho_\g(x, y)$  satisfies the triangle   inequality. It remains  to show  that  $\rho_\g(x, y) =0 $  iff   $ x =y$.
Since  constant  maps belong  to $\Dd_\Xx$,  it follows   that $\rho_g (x, x) = 0$ for all $ x\in \Pp_\Xx$.
To prove  that  $\rho_\g (x, y)  = 0$   implies  $ x= y$,    it suffices  to prove the following

\begin{lemma}\label{lem:bound1}    For  any  $x, y \in \Pp_\Xx$  we have
$$  \rho_\g  (x, y) \ge  \|  x -y \|_{TV} .$$
\end{lemma}
\begin{proof}
Let $\gamma:[a,b]\to\Pp_\Xx\subset\Ss(\Xx)$ be a  $C^k$-curve joining $x$ and $y$.
Since
$$\frac{d\dot \gamma (t)}{d\gamma (t)} \in L^2  (\gamma(t))$$
for all $t$, we have
\begin{multline*}
\|y-x\|_{TV}=\bigl\|\gamma(b)-\gamma(b)\bigr\|_{TV}=\biggl\|\int_a^b\dot\gamma(t)\,dt\biggr\|_{TV}\\
\le\int_a^b\bigl\|\dot\gamma(t)\bigr\|_{TV}\,dt\le\int_a^b\bigl|\dot\gamma(t)\bigr|_\g\,dt.
\end{multline*}
This    proves  Lemma  \ref{lem:bound1}   for  $C^k$-curves $\gamma$.

Next  we assume  that  $\gamma : [0,1] \to  \Pp_\Xx\to  \Ss (\Xx)$ is a   piece-wise   $C^k$-curve.   Combining   the previous   argument  and   the  triangle  inequality  for  the  total variation norm,        we   complete  the  proof    of  Lemma \ref{lem:bound1}  immediately.
\end{proof}
This completes  the   proof  of Theorem  \ref{thm:distance},
\end{proof}

\begin{remark}\label{rem:wdist}  Note that   Definition  \ref{def:length}  also works  for  weakly
$C^k$-diffeological statistical models, but the  proof  of  Lemma \ref{lem:bound1}  does not work for   weak   $C^1$-maps  $\gamma: [a,b] \to \Pp_\Xx \subset \Ss(\Xx)$. Since
any  weakly  differentiable map  $ \gamma: [a,b] \to \Ss(\Xx)$ is  a.e.   differentiable \cite[Theorem 3.2]{Kaliaj2016}, we conjecture  that  Lemma \ref{lem:bound1}  and      Theorem \ref{thm:distance} also  hold  for    weakly $C^k$-diffeological statistical models.
\end{remark}

Note  that  our  definition  of the diffeological Fisher  metric  and  the  diffeological Fisher distance  is   coordinate-free.
In the remainder  of this  section  we shall  show  the   naturality   of the   diffeological  Fisher metric  and  the  diffeological   Fisher  distance, using the   language  of probabilistic morphisms.

In 1962 Lawvere   proposed  a categorical approach to     probability theory,  where  morphisms  are  Markov kernels, and   most importantly, he supplied  the space $\Pp(\Xx)$  with  a natural  $\sigma$-algebra   $\Sigma_w$, making the notion of Markov kernels  and hence
many  constructions  in  probability theory  and  mathematical statistics functorial \cite{Lawvere1962}.

Let us recall the definition of $\Sigma_w$.  Given a measurable space $\Xx$,  let $\Ff_s(\Xx)$ denote
the  linear space of simple functions on $\Xx$. 
There  is a natural  homomorphism $I: \Ff_s(\Xx) \to \Ss^*(\Xx):= \Hom\bigl(S(\Xx),\R\bigr)$, $f\mapsto I_f$,
defined  by  integration:  $I_f (\mu):= \int_\Xx f d\mu$  for $f \in \Ff_s(\Xx)$ and $\mu \in \Ss(\Xx)$.
Following Lawvere  \cite{Lawvere1962},  we  define $\Sigma_w$ to be  the smallest $\sigma$-algebra  on $\Ss(\Xx)$  such that $I_f$  is measurable  for all $f\in \Ff_s (\Xx)$.   Let  $\Mm(\Xx)$ denote  the space  of all finite  nonnegative measures on $\Xx$. We also  denote  by $\Sigma_w$  the restriction  of  $\Sigma_w$
to $\Mm (\Xx)$, $\Mm^* (\Xx): = \Mm(\Xx) \setminus \{0\}$,  and  $\Pp (\Xx)$.

\begin{definition}\label{def:prob}{\cite[Definition 1]{JLT2021}} {\it A probabilistic  morphism}
(or  {\it an arrow}) 	from   a measurable space $\Xx$ to a measurable  space $\Yy$  is
an	measurable  mapping  from  $\Xx$ to  $\bigl(\Pp(\Yy),\Sigma_w\bigr)$.
\end{definition}

We shall  denote by $\overline{T}\:\Xx\to\bigl(\Pp(\Yy),\Sigma_w\bigr)$  the measurable  mapping  defining/generating a probabilistic  morphism $T: \Xx \leadsto \Yy$.
Similarly, for a measurable mapping $\pb: \Xx \to \Pp(\Yy)$  we shall denote by $\underline{\pb}: \Xx \leadsto \Yy$ the generated probabilistic  morphism.   Note that    a   probabilistic  morphism  is  denoted  by a curved  arrow  and  a measurable mapping by a straight arrow.

From  now on  we    shall  always assume  that $\Pp (\Xx)$ is a  measurable  space   with  the  $\sigma$-algebra $\Sigma_w$.  Let $\delta_x  \in \Pp (\Xx)$ denote  the  Dirac  measure  concentrated  at $x$ on  $\Xx$.
Giry proved that the inclusion  $i : \Xx \to \Pp (\Xx), \, x \mapsto \delta_x,$  is  a measurable  mapping
\cite{Giry1982}. It follows  that   any  measurable  mapping  $\kappa: \Xx \to \Yy$   assigns
a  probabilistic   morphism $\underline{i \circ  \kappa}: \Xx \leadsto \Yy$, which   we  shall write  shorthand  as $\underline{\kappa}: \Xx \leadsto \Yy$.  Hence  the set of probabilistic mappings between  $\Xx$  and $\Yy$ contains  a subset  of  measurable mappings between  $\Xx$ and $\Yy$.

Given  a    probabilistic  mapping   $T:  \Xx \leadsto \Yy$,   we define   a   linear  map   $S_*(T): \Ss(\Xx) \to \Ss(\Yy)$, called {\it Markov morphism},
as follows  \cite[Lemma 5.9, p. 72]{Chentsov1972}
\begin{equation}\label{eq:markov1}
S_*(T) (\mu) (B): = \int_{\Xx}\overline T (x) (B)d\mu (x)
\end{equation}
for any    $\mu \in \Ss (\Xx)$  and  $B \in \Sigma_\Yy$.
We also denote by $T_*$  the map  $S_*(T)$ if no  confusion can arise. It is known  that
$T_*(\Pp (\Xx) ) \subset \Pp (\Yy)$ \cite[Proposition 1]{Le2020}. Abusing notation, given  a     probabilistic mapping $ T: \Xx \leadsto \Yy$  and a   $C^k$-diffeological  statistical
model $(\Pp_\Xx, \Dd_\Xx)$ we define a    $C^k$-diffeological  space  $(T_*(\Pp_\Xx),  T_*(\Dd_\Xx))$    as the image of $(\Pp_\Xx, \Dd_\Xx)$   under   the smooth map $T_*: \Pp(\Xx) \to \Pp (\Yy)$.

Diffeological (almost/$2$-integrable) statistical models are preserved under probabilistic morphisms.

\begin{proposition}{\cite[Theorem 1]{Le2020}}\label{prop:preser}
Let $T\:\Xx\leadsto\Yy$ be a probabilistic morphism and $(\Pp_\Xx,\Dd_\Xx)$ a $C^k$-diffeological statistical model.
\begin{enumerate}
\item Then $\bigl(T_*(\Pp_\Xx),T_*(\Dd_\Xx)\bigr)$ is a $C^k$-diffeological statistical model.
\item If $(\Pp_\Xx,\Dd_\Xx)$ is an almost $2$-integrable $C^k$-diffeological statistical model, then $\bigl(T_*(\Pp_\Xx),T_*(\Dd_\Xx)\bigr)$ is also an almost $2$-integrable $C^k$-diffeological statistical model.
\item If $(\Pp_\Xx,\Dd_\Xx)$ is a $2$-integrable $C^k$-diffeological statistical model, then $\bigl(T_*(\Pp_\Xx),T_*(\Dd_\Xx)\bigr)$ is also a $2$-integrable $C^k$-diffeological statistical model.
\end{enumerate}
\end{proposition}

\begin{remark}\label{rem:winv}  Proposition \ref{prop:preser}  also holds  for  weakly $C^k$-diffeological statistical  models, because  the  transformation $T_* : \Pp(\Xx )\to \Pp (\Xx)$, where  $T :  \Xx \leadsto \Yy$  is a   probabilistic morphism,  is  the restriction of a  linear  bounded  map $T_* = S_*(T)$  from  $\Ss (\Xx)$ to itself.
\end{remark}

Furthermore, the diffeological Fisher  metric (and hence  the diffeological  Fisher  distance)   is decreasing under    probabilistic   morphisms   and       invariant under sufficient    probabilistic  morphisms. Denote by $L(\Xx)$
the space of bounded
measurable functions on $\Xx$. Recall   that   a  probabilistic  morphism $T: \Xx \leadsto \Yy$    is called {\it sufficient  for $\Pp_\Xx$}  if  there exists a  probabilistic
morphism $\underline \pb : \Yy \leadsto \Xx$ such that  for all $\mu \in \Pp_\Xx$   and  $h \in L(\Xx)$ we have (\cite[Definition 2.22]{JLT2021}, cf. \cite{MS1966})
\begin{equation}
\label{eq:suff}
T_*(h \mu) = \underline{\pb}^*(h)T_*(\mu) \text{, i.e., }  \underline{\pb} ^* (h) = \frac{d T_* (h \mu)}{d T_* (\mu)}\in  L^1(\Yy, T_*(\mu)).
\end{equation}

Examples  of   probabilistic morphisms      $T: \Xx \leadsto \Yy$  that are  sufficient  for     a statistical model   $\Pp_\Xx \subset \Xx$ are    1-1 measurable mappings  \cite[Example 20]{Le2020},
and   measurable  mappings  $\kappa: \Xx \to \Yy$    that  are ``regular" and  satisfying the
Fisher-Neymann  condition, see  \cite[Example 4]{JLT2021}, \cite[Example 19]{Le2020}  for more details.

\begin{proposition}\label{prop:mono}{\cite[Theorem 2]{Le2020}}
Let  $T: \Xx  \leadsto \Yy$ be a probabilistic  morphism  and $(\Pp_\Xx, \Dd_\Xx)$ an almost 2-integrable  $C^k$-diffeological statistical model.   Then for   any  $\mu \in \Pp_\Xx$  and  any $v \in T_\mu (\Pp_\Xx, \Dd_\Xx)$  we have the following monotonicity
	$$  \g_\mu  (v, v) \ge \g _{T_*\mu}  (T_*v, T_*v)$$
	with the equality    if  $T$ is sufficient  for   $\Pp_\Xx$.
\end{proposition}

\begin{remark}\label{rem:mono}  Proposition \ref{prop:mono} also holds  for   almost 2-integrable weakly  $C^k$-diffeological  statistical  models, since  the monotonicity  assertion follows from the  fact  that, given  a probabilistic morphism $T: \Xx \leadsto \Yy$,  the norm  of the   associated linear  bounded map $T_* : \Ss (\Xx) \to \Ss(\Yy)$ in  Remark \ref{rem:winv} is less  than or equal to 1. From  the monotonicity  assertion  we obtain the second assertion  concerning    sufficient  probabilistic  morphisms, since   if  $T: \Xx \leadsto \Yy$ is sufficient  w.r.t.  $\Pp_\Xx$  then  by \cite[Theorem 2.8.2]{JLT2021} there exists a  probabilistic  morphism
$\pb: \Yy \to \Xx$  such that $\pb_* (T_* (\Pp_\Xx))  = \Pp_\Xx$  and therefore $\pb_*(T_* (\Dd_\Xx)) = \Dd_\Xx$.
\end{remark}

The monotonicity (and  the invariance     under sufficient        probabilistic morphisms)  of   the    diffeological Fisher metric  suggests  that the    diffeological Fisher metric   can be    regarded as  information  metric  on  almost  2-integrable (weakly) $C^k$-diffeological   statistical models   cf. \cite{AJLS2015, AJLS2017, AJLS2018}, \cite{Le2017}.

\section{Diffeological    Cram\'er--Rao inequality}\label{sec:3}

For a locally convex  topological  vector  space  $V$  we    denote by  $Map(\Pp_\Xx, V)$  the space of all    mappings $\varphi:\Pp_\Xx\to V$  and by $V'$  the topological   dual of $V$.
Sometime we need   to  estimate  only  a ``coordinate"  $\varphi(\xi)$ of a   probability measure $\xi \in \Pp_\Xx$, which determines $\xi$  uniquely if  $\varphi$ is an embedding.

\begin{definition}
	\label{def:vares}{\cite[Definition 8]{Le2020}}  Let  $\Pp_\Xx$ be  a   statistical model  and $\varphi \in Map (\Pp_\Xx, V)$.
	{\it  A nonparametric $\varphi$-estimator $\hat \sigma_{\varphi}$}  is  a  composition  $\varphi\circ \hat \sigma:  \Xx \stackrel{\hat{\sigma}}{\to} \Pp_\Xx \stackrel{\varphi} {\to} V$. 
\end{definition}

\begin{example}\label{ex:reg}   (1)    In supervised  learning  with  an  input   space $\Xx$   and  a  label  space $\Yy$  we are interested  in   the stochastic relation  between $x\in \Xx$ and its label $y \in \Yy$, which is  expressed    via a  measure
$\mu \in (\Pp (\Xx \times \Yy)$    that governs  the distribution of      labelled  pair $(x, y) \in \Xx \times \Yy$.   Finding $\mu$  is a   density   estimation  problem, assuming that we  are given
a sequence  of  i.i.d.     labelled   pairs  $\{(x_1, y_1), \cdots,  (x_n, y_n)\}$.   In practice, we  are  interested   only  in  knowing  the conditional  probability  $\mu_{\Yy|\Xx} (\cdot| x)$, which is regular
under   very   general assumptions  \cite{Faden1985}.  Then    finding   the  conditional  probability
$\mu_{\Yy|\Xx} (\cdot |x)$  is  equivalent   to finding   a   measurable mapping  $\overline T: \Xx  \to \Pp (\Yy)$,  or equivalently, a probabilistic morphism $T:\Xx \leadsto \Yy$. Usually  $\Yy$    is   represented  as  a subset  in $\R^n$   and  the knowledge of $\mu_{\Yy|\Xx} (\cdot |x)$  is often not required, it is sufficient  to determine  one of its characteristics,  for  example the  regression function
$$ r_\mu(x) = \int_\Yy  y d \mu_{\Yy|\Xx}(y|x).$$
In this case, the map $\varphi: \Pp (\Xx \times \Yy) \to  Map (\Xx, \R), \mu \mapsto r_\mu,$    is defined  as the composition of the  mappings defined  above
$$ \Pp (\Xx \times  \Yy) \to {\rm  Probm} (\Xx, \Yy) \to Map (\Xx, \R),$$
where   ${\rm Probm} (\Xx, \Yy)$  denotes  the space of   probabilistic morphisms  from $\Xx$ to $\Yy$.

(2) A classical example of a $\varphi$-map  is the   moment of   a   probability measure in a  1-dimensional   statistical model  $\pb (\Theta)$, where   $\Theta$ is an  interval or the real line. Given   a  real  function $g(x)$,   we define
$$\varphi  (\pb (\theta)): =   \int  g(x) d \pb  (\theta).$$
Under a certain condition   this map is  1-1 \cite[p. 55]{Borovkov1998}.
\end{example}

Now we shall    define an admissible  class  of  $\varphi$-estimators, introduced in \cite{Le2020}.
Let  $(\Pp_\Xx, \Dd_\Xx)$ be a $C^k$-diffeological  statistical model  and $V$  a   locally convex  vector  space.
For   $\varphi \in Map(\Pp_\Xx, V)$ and $l \in V'$ we    denote by $\varphi^l$ the composition  $l \circ \varphi$.  Then  we set
$$
L^2_{\varphi}(\Xx,\Pp_\Xx):=\bigl\{\hat\sigma\:\Xx\to P_\Xx\mid\varphi^l\circ\hat\sigma\in L^2_\xi(\Xx)\text{ for all }\xi\in P_\Xx\text{ and }l \in V'\bigr\}.
$$
For   $\hat \sigma \in  L^2_\varphi(\Xx, \Pp_\Xx)$  we define  the $\varphi$-mean value    of $\hat  \sigma$, denoted  by
$\varphi_{\hat \sigma}: \Pp_\Xx \to  V'' $,  as follows (cf. \cite[(5.54), p. 279]{AJLS2017})
$$\varphi_{\hat \sigma}   (\xi) (l): = \E_\xi (\varphi^l \circ \hat \sigma) \text{  for }   \xi \in \Pp_\Xx \text{ and } l \in V', $$
where $\E_\xi$   denoted  the mathematical expectation  w.r.t.  the  probability  measure $\xi \in \Pp (\Xx)$.
Let  us identify   $V$ with a  subspace   in $V^{''}$ via the canonical pairing.

The  difference  $b^{\varphi}_{\hat{\sigma}}: = \varphi_{\hat{\sigma}}- \varphi \in Map(\Pp_\Xx,  V ^{''})$   will be called {\it the bias }  of the $\varphi$-estimator $\varphi\circ \hat \sigma$.

For all $\xi \in \Pp_\Xx$ we   define  a    quadratic function $MSE^{\varphi}_\xi [\hat \sigma]$     on $V'$, which   is
called  the {\it mean  square error  quadratic   function}  at $\xi$,   by setting  for $l, h \in V'$ (cf. \cite[(5.56), p. 279]{AJLS2017})
\begin{equation}\label{eq:qd}
\MSE^\varphi_\xi[\hat\sigma](l,h):=
\E_\xi\Bigl[\bigl(\varphi^l\circ\hat\sigma-\varphi^l(\xi)\bigr)\cdot\bigl(\varphi^h\circ\hat\sigma -\varphi^h(\xi)\bigr)\Bigr].
\end{equation}
Similarly we define   the {\it variance  quadratic  function}  of   the $\varphi$-estimator $\varphi\circ \hat \sigma$  at $\xi \in \Pp_\Xx$
is     the quadratic  form $V^\varphi_\xi [\hat \sigma]$ on $V'$ such that for all   $l, h \in V'$ we have (cf.  \cite[(5.57), p.279]{AJLS2017})
$$V^\varphi_\xi [\hat \sigma](l, h) = \E_\xi [(\varphi^l\circ \hat \sigma-\E_\xi (\varphi^l\circ \hat \sigma))\cdot ( \varphi^h\circ \hat \sigma-\E_\xi(\varphi^h\circ \hat \sigma))].$$
Then it is known that (cf.\cite[(5.58), p. 279]{AJLS2017})
\begin{equation}\label{eq:vmse}
\MSE^\varphi_\xi [\hat \sigma] (l, h) = V^\varphi_\xi [\hat \sigma](l,h) +  b^\varphi_{\hat \sigma}(\xi)(l)  \cdot  b^\varphi_{\hat \sigma} (\xi) (h).
\end{equation}

Now  we assume that  $(\Pp_\Xx, \Dd_\Xx)$ is  an  almost 2-integrable
$C^k$-diffeological   statistical model.    For any $\xi \in \Pp_\Xx$
let  $ T_{\xi}^\g (P_\Xx,\Dd_\Xx)$  be the completion of  $T_\xi (P_\Xx,\Dd_\Xx)$ w.r.t.  the     Fisher  metric $\g$. Since  $T_\xi ^\g (\Pp_\Xx, \Dd_\Xx)$  is   a Hilbert  space,  the map
$$L_\g: T_\xi^\g (\Pp_\Xx,\Dd_\Xx) \to (T_\xi^\g (\Pp_\Xx,\Dd_\Xx))',\, L_\g (v)(w) := \la  v, w\ra_\g,  $$   is an isomorphism.
Then we define  the  inverse $\g ^{-1}_\xi$ of the Fisher  metric $\g_\xi$ on $ (T_\xi^\g (\Pp_\Xx, \Dd_\Xx))'$ as  follows
\begin{equation}
\label{eq:Finv}
\g_\xi ^{-1} (L_g  v, L_g  w) : = \g_\xi ( v, w)
\end{equation}

\begin{definition}
	\label{def:reg}{\cite[Definition 9]{Le2020}, cf. \cite[Definition 5.18, p. 281]{AJLS2017}}
	Assume  that $\hat \sigma \in L^2_{\varphi}(\Xx, \Pp_\Xx)$.  We shall call  $\hat \sigma$  {\it a $\varphi$-regular estimator},
	if for all  $l \in V'$   the function  $\xi \mapsto \|\varphi^l \circ  \hat \sigma\|_{L^2 (\Xx, \xi)}$ is locally bounded, i.e.,
	for all  $\xi_0 \in \Pp_\Xx$
	$$\lim_{\xi \to \xi_0} \sup \|  \varphi^l\circ \hat \sigma  \|_{L^2 (\Xx, \xi)}  < \infty.$$
\end{definition}

For any $\xi \in \Pp_\Xx$  we
denote  by  $(\g ^\varphi_{\hat \sigma}) ^{-1}_\xi$  to be   the   following  quadratic form    on $V'$:
\begin{equation}\label{eq:g1}
(\g ^\varphi_{\hat \sigma})^{-1}_\xi(l, k) : = \g^{-1}_\xi ( d\varphi ^l_{\hat \sigma}, d\varphi^k_{\hat \sigma}) = \g_\xi ( \mathrm{grad}_\g  (\varphi^l_{\hat \sigma} ), \mathrm{grad}_\g  (\varphi^k_{\hat \sigma} )) .
\end{equation}
If  $\varphi: \Pp_\Xx \to V$ is  a local coordinate chart   around  a point $\xi\in \Pp_\Xx$  and  $\hat \sigma$ is $\varphi$-unbiased  then  $(\g^\varphi_{\hat \sigma})^{-1}_\xi$  is the   inverse  of the Fisher  metric   at $\xi$, see \cite[\S 5.2.3 (A), p. 286]{AJLS2017}.

In \cite{Le2020} L\^e proved the following diffeological  Cram\'er--Rao inequality
\begin{theorem}
\label{thm:cr} {\cite[Theorem 3]{Le2020}}  Let $(\Pp_\Xx, \Dd_\Xx)$   be a   2-integrable   $C^k$-diffeological statistical model, $\varphi$  a $V$-valued function on $\Pp_\Xx$ and  $\hat \sigma \in L^2_{\varphi} (\Xx, \Pp_\Xx)$  a $\varphi$-regular  estimator.  Then the
difference $\V_{\xi}^\varphi[\hat \sigma] -   (\hat \g^{\varphi}_{ \hat \sigma})^{-1}_\xi$     is   a  positive semi-definite  quadratic form on $V'$  for any $\xi \in \Pp_\Xx$.
\end{theorem}

\begin{remark}\label{rem:cr}  The  proof   of   Theorem \ref{thm:cr}  does not   extend    to the case of  2-integrable weakly $C^k$-diffeological  statistical models $(\Pp_\Xx, \Dd_\Xx)$.  The  main  problem is  the validity  of the  differentiation  under  integral sign  for  a $C^k$-curve  $c: (0,1) \to (\Pp_\Xx, \Dd_\Xx),\:  t \mapsto \mu_t,$
\begin{equation}\label{eq:diff}
{d\over  dt}\int _\Xx l \circ \varphi \circ \hat \sigma\, d \mu_t = \int_\Xx l \circ \varphi \circ \hat \sigma\,d\mu_t'
\end{equation}
where  $\mu_t' = \p ^w _t c (t)$.  This identity  is valid   if  $i \circ c: (0,1) \to \Ss(\Xx)$ is a   $C^1$-map   and if the function  $\xi \to \| \varphi^l \circ \hat \sigma\|_{L^2 (\Xx, \xi)}$ is  locally bounded, see \cite[Lemma 5.2, p.  282]{AJLS2017}, whose proof involves estimations  using      the  total  variation norm. This  local boundedness condition has been stated in Definition \ref{def:reg} and  Theorem \ref{thm:cr}.   The identity (\ref{eq:diff})  has been used  in the proof of   \cite[Proposition 2]{Le2020},  which is an important  ingredient  of the   proof  of the    diffeological  Cram\'er--Rao inequality \cite[Theorem 3]{Le2020}. If  instead  of the  condition that  the function  $\xi \to \| \varphi^l \circ \hat \sigma\|_{L^2 (\Xx, \xi)}$ is  locally bounded,   we   assume  a stronger condition that
$l\circ \varphi \circ \hat \sigma: \Xx \to \R$ is a bounded   function for all $l \in V'$,  by Remark \ref{rem:weak}   the   identity  (\ref{eq:diff}) holds.   Under this  stronger assumption,  the Cr\'amer-Rao equality     holds  for  2-integrable weakly $C^k$-diffeological  statistical models $(\Pp_\Xx, \Dd_\Xx)$, since     other  arguments used in the   proof  of  the diffeological  Cr\'amer-Rao inequality  \cite[Theorem 3]{Le2020}   also hold for     this general case. We conjecture  that Theorem \ref{thm:cr} is also valid  for  weakly $C^k$-diffeological  statistical  models, since  any  weakly $C^1$-map $[0,1] \to \Ss(\Xx)_{TV}$  is a.e.   differentiable  by \cite[Theorem 3.2]{Kaliaj2016}.
\end{remark}

\section{Diffeological Hausdorff--Jeffrey   measure}\label{sec:hausdorfj}

In the  previous sections   we demonstrated   that  the    diffeological  Fisher    metric    is  a  natural  extension of  the Fisher    metric,  and   the diffeological Fisher  metric  plays the same  role   in  frequentist  nonparametric estimation as the Fisher   metric  in    frequentist   parametric estimation.     In this  section   we  shall    introduce  the concept  of the Hausdorff--Jeffrey measure,   using the     diffeological Fisher metric  and the  concept  of the Hausdorff measure, which plays a fundamental role in geometric  measure  theory \cite{Federer1969}.

Let us   first  recall  the concept   of   the Hausdorff measure on   a  metric  space  $(E, d)$, following \cite{Federer1969}, \cite{AT2004}, \cite{Morgan2009}.
Recall   that for any  subset  $S \subset E$   {\it  the diameter}  of  $S$  is
$$ \diam \,  (S)  = \sup \{ d (x, y)|\, x, y \in S\}.$$
For  $k \in \N$    let   $\alpha_k$ denote the   Lebesgue   measure  of the   closed  unit  ball
$B^k(0,1)$ of radius  1 and centered at  $0$      in $\R^k$.
Let $A\subset  E$. For small  $\delta$  cover $A$ efficiently  by countably many sets $A_j$ with  $\diam \, (A_j) \le \delta$,  and  {\it the $k$-dimensional  Hausdorff  measure  of $A$}  is defined as follows
\begin{equation}\label{eq:hd}
\Hh^k(A):=\lim_{\delta\to0}\alpha_k\inf\biggl\{\sum_{j\in I}\bigg(\frac{\diam\,(A_j)}{2}\bigg)^k\mid\diam\,(A_j)
\le\delta\,\&\,A\subset\bigcup_{j\in I}A_i\biggr\}.
\end{equation}
It is known that  $\Hh^k$ is a  regular Borel measure \cite[p. 171]{Federer1969}, see also \cite[Theorem  2.1.4, p.21]{AT2004}. Furthermore, $\Hh^0$ is the  counting  measure.

The  definition  of the $k$-dimensional Hausdorff  measure extends  to any nonnegative  real dimension $k$, by  extending  the  definition of  $\alpha_k$  with  the following  definition
$$ \alpha_k : = \frac{\pi ^{k/2}}{\Gamma  (1 + k/2)} \text { where }  \Gamma (t) : = \int _0 ^\infty  x^{t-1} e^{-x}\, dx.$$

The {\it  Hausdorff dimension}  $\Hh\mhyp\dim (A)$ of a  nonempty  set $A\subset (E,d)$  is defined as
\begin{center}
$\Hh\mhyp\dim (A) : =\inf \{ m \ge 0 |\,  \Hh^m (A) = 0\}$.
\end{center}
It is known that if $k > \Hh\mhyp\dim  (A)$  then $\Hh ^k (B) = 0$ and
 if  $k < \Hh\mhyp\dim (A)$  then $\Hh^k (A)  = \infty$.

The Hausdorff measure  enjoys the following  natural  properties.
\begin{proposition}\label{prop:ambr1}

{\rm(1)}  Let $(M^m, g)$  be  a   Riemannian    manifold, regarded  as a metric  space  with  the Riemannian  distance $d_g$. Then  the  Hausdorff  measure  $\Hh^m$ on $M^m$  coincides  with the standard volume.

{\rm(2)}  Let $\varphi: A \subset (M^k, g) \to (N^n, g')$   be a Lipschitz  map  from an open domain    $A$ in a Riemannian manifold  $(M^k, g)$  of dimension $k$ to a Riemannian manifold   $(N^n, g')$  of dimension $n$ and   $n \ge k$.  By Rademacher's  theorem  $d\varphi$ and  its {\it area   factor }
$$
\mathbf{J}\,d\varphi:=\sqrt{\det\bigl((d\varphi)^*\circ(d\varphi)\bigr)}
$$
 are  defined $\Hh^k$-almost everywhere   on $A$. If  $k  =n$ then  we have  the following {\it area  formula}
\begin{equation*}
\Hh^n (\varphi(A)) = \int_A {\bf J} d\varphi d\Hh^n (x).
\end{equation*}
\end{proposition}
For  a  proof of  Proposition \ref{prop:ambr1}(1) see  \cite[p. 29,30]{AT2004}.  For a  proof  of  Proposition \ref{prop:ambr1}(2)  see   \cite[p. 44, 45]{AT2004}.

\begin{definition}\label{def:hausdorfj}  Let $(\Pp_\Xx, \Dd_\Xx)$ be a    2-integrable  $C^k$-diffeological  statistical model   with  the diffeological  Fisher   distance  $d_\g$  and $m\in \R$  the Hausdorff dimension   of $(\Pp_\Xx, d_\g)$. Then  the  Hausdorff  measure $\Hh^m_\g$ on  $(\Pp_\Xx, d_\g)$
will  be called  the  {\it diffeological Hausdorff--Jeffrey measure of $(\Pp_\Xx, \Dd_\Xx)$}.
\end{definition}

Now   we shall relate  the  diffeological Hausdorff--Jeffrey  measure  $\Hh^m_\g$  with  the  unnormalized    Jeffrey   prior  measure  $\Jj ^m_\g$  defined   on a     2-integrable  parameterized  statistical  model  $(M^m, \Xx, \pb)$, where  $\pb: M^m \to \Ss (\Xx)$  is an injective   $C^1$-map \cite{Jeffrey1946}.  Recall that $J_\g$ is   equal  to the Riemannian
volume  of the (possibly degenerate)  Riemannian manifold  $(M^k, \g)$,  whose density is   zero  at the    points  of $M$ where  the Fisher  metric $\g$ is degenerate.
\begin{theorem}\label{thm:jeffrey}
{\rm(1)} Let $(M^m,\Xx,\pb)$ be a $2$-integrable parameterized statistical model, where $M^m$ is a smooth manifold of dimension $m$ and $\pb\:M^m\to\Ss(\Xx)$ is an injective $C^1$-map. We regard $\pb$ as a $C^1$-map from $M^m$ to the $2$-integrable $C^1$-diffeological statistical model $\bigl(\pb(M),\pb_*(\Dd_{can})\bigr)$. Then
$$
\pb_*(\Jj^m_\g)=\Hh^m_\g.
$$
{\rm(2)} Let $T\:\Xx\leadsto\Yy$ be a probability morphism and $(\Pp_\Xx,\Dd_\Xx)$ a $2$-integrable $C^k$-diffeological statistical model. Then for any $k\in\R$ and any Borel set in $(\Pp_\Xx,d_\g)$ we have
\begin{equation}\label{eq:hausdorfj}
\Hh^k_\g (A)\ge\Hh^k_\g\bigl(\pb_* (A)\bigr).
\end{equation}
The inequality in $(\ref{eq:hausdorfj})$ turns to inequality if $T$ is sufficient w.r.t. $\Pp_\Xx$.
\end{theorem}
\begin{proof}
The  first assertion  of  Theorem   \ref{thm:jeffrey}  follows    from Proposition \ref{prop:ambr1}.  The  second  assertion   is a   consequence of  Proposition \ref{prop:mono}.
\end{proof}
According to Jordan \cite{Jordan2011},  to  justify   the choice  of  an a prior  probability measure
 is  one  of  main    theoretical   problems   in Bayesian  statistics.   Theorem
\ref{thm:jeffrey} justifies  the choice  of the Hausdorff--Jeffrey  measure  as natural objective a prior  measure on 2-integrable   $C^k$-diffeological statistical models.

\section{Conclusion  and  outlook}
In this  paper  we showed  that the concept  of the diffeological  Fisher  metric is   a natural  extension of the  notion of the Fisher metric,  originally defined  on parameterized  statistical models.    There are two advantages
of the  nonparametric  diffeological   Fisher metric:  (1) it can be defined  on singular  statistical  models, (2)  it    turns     a   2-integrable   $C^k$-diffeological  statistical  model  to a length space, on which  the  Hausdorff measure  is  a   natural  generalization  of the Jeffrey measure.   We also discussed  some  open  questions concerning   extending results from   $C^k$-diffeological   statistical   models   to  weakly $C^k$-diffeological statistical models.  To make  more use  of  the diffeological   Fisher metric  we  expect   that we need  to  put  certain assumptions on the  singularities  of
underlying  2-integrable  $C^k$-diffeological statistical models.  In the case  a  $C^k$-diffeological  statistical model  does not admit a  diffeological  Fisher metric, we might  consider  instead  diffeological   Finsler metric as in \cite{Amari1984}. In view of  recent    developments  of   Barbaresco's and  Gay-Balmaz'  geometric theory of  Gibbs  probability  densities  with promising applications in machine  learning \cite{BG2020},  we  plan to  develop  a  theory of  diffeological   exponential  models for  a  diffeological treatment of  infinite dimensional  families  of Gibbs   probability  densities.

\section*{Acknowledgement} HVL  would like to thank Professor  Fr\'ed\'eric Barbaresco  and Professor Frank Nielsen    for   their invitation to the
Les Houches conference  in July 2020 where   a  part of the  results  in this  paper  has been  reported.  She  is     grateful  to  Professor  Sun-ichi Amari  who  discussed  with her   the phenomena of
degeneration  and explosion of the Fisher metric and   sent  her  a  copy of   his paper \cite{Amari1984} several years ago. The  authors would like to thank  the referee for    helpful comments  which  considerably improve  the  exposition of the  present paper.

\end{document}